\documentclass{amsart}
\usepackage{amssymb,euscript,amsmath, mathrsfs}
\usepackage{tikz}
\usetikzlibrary{matrix,arrows,decorations.pathmorphing}
\usepackage{tikz-cd}

\newcounter{ENUM}
\newcommand{\itm}{\item}
\newenvironment{ilist}[1][0]{\renewcommand{\theENUM}{\roman{ENUM}}\renewcommand{\itm}{\addtocounter{ENUM}{1}\item[(\theENUM)]}\begin{itemize}\setcounter{ENUM}{#1}}{\end{itemize}}
\newenvironment{Ilist}[1][0]{\renewcommand{\theENUM}{\Roman{ENUM}}\renewcommand{\itm}{\addtocounter{ENUM}{1}\item[(\theENUM)]}\begin{itemize}\setcounter{ENUM}{#1}}{\end{itemize}}

\newcommand{\margh}[1]{}

\input xy
\xyoption{all}
\CompileMatrices

\def\AA{{\mathbb A}}

\def\cU{{\mathcal U}}

\def\cX{{\mathcal X}}
\def\cY{{\mathcal Y}}
\def\cZ{{\mathcal Z}}

\def\GL{\operatorname{GL}}

\def\Spec{\operatorname{Spec}}

\def\id{\operatorname{id}}

\def\codim{\operatorname{codim}}

\def\Gm{G_m}

\def\reldim{\operatorname{reldim}}

\newtheorem{thm}{Theorem}[section]
\newtheorem{prop}[thm]{Proposition}
\newtheorem{lem}[thm]{Lemma}
\newtheorem{cor}[thm]{Corollary}

\theoremstyle{definition}
\newtheorem{defn}[thm]{Definition}

\newtheorem{ex}[thm]{Example}

\theoremstyle{remark}
\newtheorem{notn}[thm]{Notation}
\newtheorem{rem}[thm]{Remark}
\newtheorem{rems}[thm]{Remarks}

\numberwithin{equation}{section}
\numberwithin{figure}{section}

\begin{document}
\title{Relative dimension of morphisms and dimension for algebraic stacks}
\author{Brian Osserman}
\begin{abstract} 
Motivated by applications in moduli theory, we introduce a flexible and
powerful language for expressing lower bounds on relative dimension of 
morphisms of schemes, and more generally of algebraic stacks. We show that 
the theory is robust and applies to a wide range of situations. Consequently, 
we obtain simple tools for making dimension-based deformation arguments on
moduli spaces. Additionally, in a complementary direction we develop the 
basic properties of codimension for algebraic stacks. One of our goals is 
to provide a comprehensive toolkit for working transparently with dimension 
statements in the context of algebraic stacks.
\end{abstract}

\thanks{The author was partially supported by NSA grant H98230-11-1-0159
during the preparation of this work.}

\maketitle

\section{Introduction}

The notion of dimension for schemes is poorly behaved, even for 
relatively basic examples such as schemes smooth over the spectrum of a 
discrete valuation ring. Related concepts which are better behaved are 
codimension and dimension of local rings. Thus, if one wants to translate 
naive dimension-based arguments from the context of schemes of finite type 
over a field to a more general setting, a standard approach is to rephrase 
results using one of these two alternatives. The purpose of the present
paper is twofold: first, to introduce a more natural way of working with
relative dimension of morphisms, and second, to generalize this -- as well
as basic properties of codimension -- to algebraic stacks.

More specifically, we introduce in Definition \ref{def:rel-dim} below
a precise formulation of what it should mean for
a morphism to have relative dimension at least a given number $n$. Our
immediate motivation is that in certain moduli space constructions in 
Brill-Noether theory (particularly the limit linear series spaces
introduced by Eisenbud and Harris), a key property of the moduli space it 
that it has at least a certain dimension relative to the base. This allows 
the use of deformation arguments based purely on dimension counts. However, 
the existing language to express these ideas is notably lacking when 
the base is not of finite type over a field. Our definition gives very 
natural language to capture what is going on, and we show that it is 
formally well behaved, occurs frequently, and has strong consequences of 
the sort that one wants for moduli theory.

In generalizing limit linear series to higher-rank vector bundles, it is
natural to work not with schemes but with algebraic stacks, and in this
context, it is even more difficult to give transparent relative dimension 
statements using the usual tools. In particular, there is no good notion of 
dimension of local rings of stacks (see Example \ref{ex:a1-mod-gm} below).
Thus, the fact that our language generalizes readily to the stack context
makes it especially useful, and it is incorporated accordingly into
\cite{os20} and \cite{o-t2}. 

Finally, as a complement to the theory of relative dimension, we also
develop the theory of codimension in the context of stacks. Here, the
statements are as expected, but we emphasize that the context of stacks
introduces certain subtleties which demand a careful treatment. These 
arise in large part because of the tendency of smooth (and even etale)
covers to break irreducible spaces into reducible ones. For instance, it
is due to these phenomena that the condition of being universally catenary 
does not descend under etale morphisms, in general.

The paper is organized as follows: in \S \ref{sec:prelims}, we recall
some background definitions and results of particular relevance. In
\S \ref{sec:rel-dim}, we give the definition of a morphism having
relative dimension at least $n$, which forms the heart of the paper. 
We analyze the basic properties of this definition, as well as 
consequences of the defined properties, with an eye towards moduli theory.
In \S \ref{sec:rel-dim-comps}, we analyze the formal behavior of our
theory, particularly with respect to compositions of morphisms, and in
\S \ref{sec:rel-dim-exs} we give a number of examples as well as further
discussion of the theory. In \S \ref{sec:stack-codim} we develop the theory 
of codimension of stacks. Finally, in \S \ref{sec:rel-dim-stacks} we 
generalize the notion of relative dimension to the context of stacks.

\subsection*{Acknowledgements} 
I would like to thank Johan de Jong for helpful conversations.

\section{Preliminaries}\label{sec:prelims}

Recall that if $Z \subseteq X$ is a closed subscheme, then
we define
\begin{align*} \codim_X Z & = \min_{Z' \subseteq Z} \codim_X Z'\\
& = \min_{Z' \subseteq Z} \left(\max_{X' \supseteq Z'} \codim_{X'} Z'\right),
\end{align*}
where we can take $Z'$ and $X'$ to run over irreducible
components of $Z$ and $X$ respectively, or over all irreducible closed
subschemes.
If $X$ is locally Noetherian, this is always finite. 

We will work extensively with universally catenary schemes. Recall that
schemes are only required to be locally Noetherian in order to be
universally catenary, and a scheme locally of finite type over a 
universally catenary scheme remains universally catenary (see Remark 5.6.3
(ii) and (iv) of \cite{ega42}). The distinction
between finite type and locally finite type is particularly important in
the context of stacks, where naturally arising irreducible stacks 
may be locally of finite type without being of finite type.

One of the primary ways in which the universally catenary hypothesis will 
arise is the ``dimension formula,'' which we now recall in the particular 
form which we will use. 

\begin{prop}\label{prop:dim-formula} Let $Y$ be a universally catenary 
scheme, and $f:X \to Y$ a morphism locally of finite
type. Let $Z' \subseteq Z$ be irreducible closed subschemes of $X$, 
and let $\eta'$ and $\eta$ be their respective
generic points. Then
$$\codim_{Z} Z'-\codim_{\overline{f(Z)}} \overline{f(Z')}
= \dim Z_{f(\eta)}- \dim Z'_{f(\eta')}.$$
\end{prop}

In the above, and throughout the paper, the subscripts are used to denote
the relevant fibers. 

Proposition \ref{prop:dim-formula} is just a rephrasing of Proposition 
5.6.5 of \cite{ega42}.

We will also make frequent use of Chevalley's theorem on upper 
semicontinuity of fiber dimension for morphisms locally of finite type; 
this is Theorem 13.1.3 of \cite{ega43}, and holds without any further 
hypotheses.

We use the term ``algebraic stack'' in the sense of Artin, following the 
conventions of \cite{l-m-b}, Chapter 4; 
in particular, their quasiseparation hypotheses imply that the underlying
topological spaces of arbitrary algebraic stacks still have the property
that every irreducible closed subset has a unique generic point (see
Corollory 5.7.2 of \cite{l-m-b}).

A \textbf{smooth presentation} of an algebraic stack $\cX$ consists
of a scheme $X$ together with a smooth, surjective morphism $P:X \to \cX$.
A \textbf{smooth presentation} of a morphism $\cX \to \cY$ of algebraic
stacks consists of schemes $X,Y$ together with smooth surjective morphisms
$P: Y \to \cY$ and $P':X \to \cX \times_{\cY} Y$. This thus induces a 
morphism $X \to Y$.

We will frequently use the fact that smooth morphisms of stacks (and
in particular smooth presentations) are open, and hence allow for lifting
of generizations; see Proposition 5.6 and Corollary 5.7.1 of \cite{l-m-b}.

\begin{notn} If $f$ is a smooth morphism of schemes or algebraic stacks,
we write $\reldim_x f$ for the relative dimension of $f$ at a point
$x$ of the source.
\end{notn}

An example of de Jong (see Tag 0355 of \cite{stacks-proj})
shows that the universally catenary condition does not
descend under etale morphisms, so it does not make sense to say that
an algebraic stack is universally catenary. However, a natural
substitute is to work with algebraic stacks locally of finite type 
over a universally catenary base scheme. We will show that such stacks
behave much like universally catenary schemes.

Finally, we make some definitions and basic observations relating to fibers 
of morphisms of stacks. Even though the fibers themselves are not in general 
well defined, properties such as dimension and geometric reducedness behave 
well.

\begin{defn}\label{def:stack-reduced-fibers} 
Let $f:\cX \to \cY$ be a morphism of algebraic stacks. For $x \in X$, let
$y$ be a geometric point representing $f(x)$, and $\widetilde{x}$ a point
of $\cX \times_{\cY} y$ lying over $x$. We say that $f$ is \textbf{reduced}
at $x$ if $\cX \times_{\cY} y$ is reduced at $\widetilde{x}$.
\end{defn}

In general, if $y'$ and $\widetilde{x}'$ are different choices of a geometric
realization of the fiber as in the definition, then there exists a 
common realization simultaneously extending $y$ and $y'$, and with a
point mapping to $\widetilde{x}$ and $\widetilde{x}'$.
In particular, because geometric reducedness is invariant under field 
extension, we have:

\begin{prop}\label{prop:reduced-fibers-stacks-defnd} The definition of
$f$ being reduced at $x$ is independent of the choices of $y$ and
$\widetilde{x}$.
\end{prop}

We now treat fiber dimension.

\begin{defn}\label{def:fiber-dim-stacks} Let $f:\cX \to \cY$ be a 
locally finite type morphism of algebraic stacks. For $x \in X$, let
$y$ be a geometric point representing $f(x)$, and $\widetilde{x}$ a point
of $\cX \times_{\cY} y$ lying over $x$, and define
$\delta_x f$ to be the dimension at $\widetilde{x}$ of $\cX \times_{\cY} y$.
\end{defn}

We then conclude:

\begin{prop}\label{prop:fiber-dim-stacks} The definition
of $\delta_x f$ is independent of the choices of $y$ and $\widetilde{x}$,
and can be expressed equivalently as follows: if $P: Y \to \cY$ and
$P':X \to \cX \times_{\cY} Y$ give a smooth presentation of $f$, and
we have $\widetilde{x} \in X$ and $\widetilde{y} \in Y$ mapping to $x$ and 
$f(x)$ respectively, then
$$\delta_x f = \dim_{\widetilde{x}} X_{\widetilde{y}} 
- \reldim_{\widetilde{x}} P'.$$

In addition, the function $x \mapsto \delta_x f$ is upper semicontinuous.
\end{prop}

\begin{proof} The first statement follows just as for reducedness, since 
dimension is invariant under field extension. Next, let $P$ and
$P'$ be a smooth presentation of $f$, and let $y$ be a geometric point
extending $\widetilde{y} \in Y$. Then $P'$ induces a smooth cover
of $\cX \times_{\cY} y$ as well, so we conclude the desired identity 
directly from the definition of dimension of an algebraic stack.

Finally, for semicontinuity, we wish to show that for all $n$,
the subset of $\cX$ on which $\delta_x f \leq n$ is open. 
Let $g:X \to Y$ be the morphism induced by $P,P'$; this is locally of 
finite type, so by semicontinuity
of fiber dimension, the locus on $X$ where the fiber dimension of $g$ is at 
most $m$ is open for all $m$. But by the previous statement of the 
proposition, $\delta_x f$ is obtained as
$$\dim_{\widetilde{x}} X_{g(\widetilde{x})} - \reldim_{\widetilde{x}} P'$$
for some $\widetilde{x} \in X$ lying over $x$, and since the relative dimension
of $P'$ is locally constant, we find that the preimage in $X$ of the locus
of $\cX$ on which $\delta_x f \leq n$ is open, and hence 
the locus itself is open.
\end{proof}

\section{Relative dimension for morphisms}\label{sec:rel-dim}

In this section, we define relative dimension, give some examples, and 
investigate some basic properties, including the ``smoothing'' results 
(Propositions \ref{prop:rel-dim-smoothing} and \ref{prop:rel-dim-smoothing-2})
which constitute the primary application of the theory.

The definition is the following:

\begin{defn}\label{def:rel-dim} Let $f:X \to Y$ be a morphism locally of 
finite type of universally catenary schemes. We say that $f$ has 
\textbf{relative dimension at least $n$} if for any 
irreducible closed subscheme $Y'$ of $Y$, and any
irreducible component $X'$ of $X|_{Y'}$, with generic point $\eta$, we have
\begin{equation}\label{eq:rel-dim}
\dim X'_{f(\eta)} - \codim_{Y'} \overline{f(X')} \geq n.
\end{equation}

We say that $f$ has \textbf{universal relative dimension at least $n$}
if for all universally catenary $Y$-schemes $S$, the base change
$X \times_Y S \to S$ has relative dimension at least $n$.

We say that $f$ has \textbf{relative dimension} (respectively, 
\textbf{universal relative dimension}) \textbf{at least $n$} at a point 
$x \in X$ if there is an open
neighborhood $U$ of $x$ such that the induced morphism $U \to Y$ has
relative dimension (respectively, universal relative dimension) at 
least $n$.
\end{defn}

\begin{rems} (i) The terminology is justified by the observation that if $Y$ 
is of finite type over a field, then \eqref{eq:rel-dim} is equivalent to 
$\dim X'-\dim Y' \geq n$. Although the latter expression is independent
of $f$, the property of having relative dimension at least $n$ may still 
depend on $f$ in general; see Remark \ref{rem:map-dependence}.

(ii) If $f:X \to Y$ has relative dimension at least $n$, then we see
immediately from the definition that every irreducible component of every
fiber of $f$ must have dimension at least $n$; see Proposition 
\ref{prop:rel-dim-prelim} below for a stronger version of this statement.

(iii) Note that here we allow $n$ to be negative, and even then the condition 
is not vacuous; see Example \ref{ex:closed-imm}.

(iv) Because dimension theory only works well for
universally catenary schemes, it is natural to restrict to this case.
We will use the hypothesis primarily to analyze composition of morphisms.

(v) Universally catenary schemes are by definition closed under 
morphisms which are locally of finite type, so we have in particular that if 
a morphism has universal relative dimension at least $n$, then every locally 
finite type base change has relative dimension at least $n$.

(vi) Although fiber dimension is upper semicontinuous, the same is not
true for our definition of relative dimension. See Example 
\ref{ex:normalization} below.

(vii) We are primarily interested in the global case; we use the local
concept of relative dimension mainly to simplify the treatment of stacks,
where smooth presentations may have different relative dimensions at
different points.
\end{rems}

We begin with two classes of examples of morphisms satisfying Definition
\ref{def:rel-dim}.

\begin{ex}\label{ex:flat} If $f$ is universally open with every generic 
fiber having every component of dimension at least $n$ (in particular, if 
$f$ is smooth of relative dimension $n$ or larger), then $f$ has universal 
relative dimension at least $n$. Indeed, because of semicontinuity of
fiber dimension, the condition is preserved under base change (including 
under restriction to irreducible subschemes of the base), so this follows
immediately from the fact that open morphisms map generic points to 
generic points.
\end{ex}

\begin{ex}\label{ex:closed-imm} For closed immersions, the first
difference between relative dimension and codimension is that codimension
is defined to be $c$ if the minimum codimension over all irreducible 
components of the source is equal to $c$, while the relative dimension 
is at least $-c$ only if the maximum codimension over all irreducible
components of the course is (at most) $c$. More generally, for closed
immersions relative dimension measures intersection codimension behavior.

In particular, if $f:X \to Y$ is a closed immersion, 
$Y$ is regular, and every component of $X$ has codimension at most $c$ in
$Y$, it is a theorem of Serre (Theorem V.3 of \cite{se4}) 
that $f$ has relative dimension at least $-c$, and building
on Serre's theorem Hochster proved in Theorem 7.1 of \cite{ho1}
that $f$ has universal relative dimension at least $-c$.

In a different direction, if $Y$ is universally catenary, and everywhere 
locally $X$ is cut out by at most $c$
equations in $Y$, then $f$ has universal relative dimension at least $-c$
by Krull's principal ideal theorem.
More generally, if $X$ is everywhere locally cut out by the 
vanishing of the $(k+1)\times(k+1)$ minors of an $n \times m$ matrix, then
$f$ has universal relative dimension at least $-(n-k)(m-k)$ (see 
Exercise 10.4 of \cite{ei1}).
Further, because $Y$ is catenary, and such descriptions are preserved
under restriction, we see that the same also holds for intersections of 
closed subschemes of the above forms, if we sum the corresponding relative 
dimensions. 

Finally, we note that since our definitions are blind to non-reduced
structure, it is in fact enough for the above descriptions of $X$ to hold
up to taking reduced structures. In particular, if $Y$ has pure dimension
$c$ at a closed point $x$, then setting $X=\{x\}$ we find that
$X \hookrightarrow Y$ has universal relative dimension at least $-c$.
\end{ex}

We next make some preliminary observations:

\begin{prop}\label{prop:rel-dim-prelim} If $f:X \to Y$ has relative dimension
at least $n$, then for any irreducible closed subset $Y'$ of $Y$, and
any irreducible component $X'$ of $f^{-1}(Y')$, for every
$x \in X'$ we have
$$\dim_x X'_{f(x)}-\codim_{Y'} f(X') \geq n.$$
\end{prop}

\begin{proof} Observing that in Definition \ref{def:rel-dim}, we
have $X'_{f(\eta)}$ irreducible and hence equidimensional, the statement 
follows immediately from semicontinuity of fiber dimension.
\end{proof}

\begin{prop}\label{prop:rel-dim-opens} If $f:X \to Y$ has relative dimension 
at least $n$, and $U \subseteq X$ is a nonempty open subset, then the
induced morphism $U \to Y$ has relative dimension at least $n$. 
Conversely, if $\{U_i\}$ is an open cover of $X$ and each morphism
$U_i \to Y$ has relative dimension at least $n$, then $X \to Y$ has
relative dimension at least $n$.

The same statements hold for universal relative dimension.
\end{prop}

Consequently, $f$ has (universal) relative dimension at least $n$ if 
and only if it has (universal) relative dimension at least $n$ at every
point of $X$.

\begin{proof} The first statement is clear from the definition, 
since $X'_{f(\eta)}$ is irreducible and of finite type over a field, so 
its dimension doesn't change when restricting to open subsets.

The second statement follows similarly: if $\{U_i\}$ is an open cover
of $X$, then some $U_i$ contains the generic point $\eta$ of $X'$, 
and both the dimension and codimension in \eqref{eq:rel-dim} are
unchanged by restriction to $U_i$.

The universal case follows immediately, since open subsets (respectively,
open covers) are preserved under base change.
\end{proof}

We conclude this section with two propositions describing the main
applications of Definition \ref{def:rel-dim}; in the context of
moduli spaces, they constitute ``smoothing theorems'' based on dimension
counts and analysis of special fibers. The strongest statements occur in the
case of nonnegative relative dimension.

\begin{prop}\label{prop:rel-dim-smoothing} Given $f:X \to Y$, suppose there 
exists $x \in X$ such that $f$ has relative dimension at least $n$ at $x$, 
and the fiber 
$X_{f(x)}$ has dimension $n$ at $x$. Then there exists a neighborhood $U$ of 
$x$ on which $f$ has relative dimension at least $n$ and pure fiber 
dimension $n$, and on any such neighborhood $f$ is open.

If further $f$ has universal relative dimension at least $n$ at $x$, then 
$f$ is universally open on $U$. 

If further $Y$ is reduced and the fiber of $f$ is geometrically reduced
at $x$, then $f$ is flat at $x$.
\end{prop}

Note that this proposition can also be viewed as a complement to the
standard criterion for flatness in terms of fiber dimension in the 
case of a Cohen-Macaulay scheme over a regular scheme.

\begin{proof} By definition, there exists an open neighborhood of $x$
on which $f$ has relative dimension at least $n$. Applying Proposition 
\ref{prop:rel-dim-prelim} and semicontinuity of fiber dimension, there 
exists a possibly smaller open neighborhood $U$ of $x$ on which $f$ also
has pure fiber dimension $n$. We claim that $f$ is necessarily open on $U$.
Let $U'$ be any non-empty open neighborhood of $U$; by 
Proposition \ref{prop:rel-dim-opens}, we have that $f$ has relative 
dimension at least $n$ and pure fiber dimension $n$, so it is enough to
show that the image of such a morphism is open. Since the morphism is
locally of finite type, by Corollary 1.10.2 of \cite{ega41} it is enough 
to show that the image is closed under generization.
Let $y$ specialize to $y'$ in $Y$, with $y'$ in the image of $f$, say
$y'=f(x')$. Set $Y'$ to be the closure of $y$, and $X'$ any irreducible
component of $f^{-1}(Y')$ containing $x'$. Then from the definition of
having relative dimension at least $n$, together with the hypothesis that
the fibers have dimension $n$, we conclude that $X'$ dominates $Y'$, and
in particular $y$ is in the image of $f$, as desired.

Next, if $f$ has universal relative dimension at least $n$, we conclude
that $f$ is universally open on $U$ because according to Corollary 8.10.2
of \cite{ega43} we may check that a morphism is universally open after 
finite type base change.

Finally, if the fiber of $f$ is geometrically reduced at $x$, and $Y$
is reduced, flatness of $f$ follows from Theorem 15.2.2 of \cite{ega43}.
\end{proof}

It is clear that Proposition \ref{prop:rel-dim-smoothing} will never
apply in the case of negative relative dimension. Because it is sometimes 
important for applications, we also state a weaker version which works 
even when the relative dimension is negative.

\begin{prop}\label{prop:rel-dim-smoothing-2} Given $f:X \to Y$, suppose 
that $Y$ is irreducible,
and there exist $x \in X$ and $Y' \subseteq Y$ closed and irreducible 
containing $f(x)$ and with support strictly smaller than $Y$,
such that 
\begin{Ilist}
\itm $f$ has relative dimension at least $n$ at $x$;
\itm every irreducible component $X'$ of $f^{-1}(Y')$ containing $x$ has
$$\dim X'_{f(\eta)} - \codim_{Y'} \overline{f(X')} = n,$$
where $\eta$ is the generic point of $X'$. 
\end{Ilist}
Then for every irreducible component $X''$ of $X$ containing $x$, we have
$$f(X'') \not \subseteq Y'.$$
If further we have
\begin{Ilist}[2]
\itm $Y'$ has codimension $c$ in $Y$, and the inclusion 
$Y' \hookrightarrow Y$ has universal relative dimension at least $-c$,
\end{Ilist}
then we have
$$\dim X''_{f(\eta)} - \codim_{Y} \overline{f(X'')} = n,$$
where $\eta'$ is the generic point of $X''$.
\end{prop}

\begin{proof} 
Given $X''$ an irreducible component of $X$ containing $x$, let $\eta'$ be 
its generic point. By (I), we have
$$\dim X''_{f(\eta')} - \codim_{Y} \overline{f(X'')} \geq n.$$
If we had $f(X'') \subseteq Y'$, then $X''$ would also
be an irreducible component of $f^{-1}(Y')$, so by (II) we would have
$$\dim X''_{f(\eta')} - \codim_{Y''} \overline{f(X'')} = n.$$
But (in light of our running catenary hypotheses) this contradicts the
hypothesis that $\codim_{Y} Y' >0$.
Thus the first assertion holds.

Next suppose that (III) is also satisfied, and let $X'$ be an irreducible 
component of $f^{-1}(Y')$ 
containing $x$ and contained in $X''$, and $\eta$ its generic point. 
Then, since we already have the opposite inequality, we want to show that
$$\dim X''_{f(\eta')} - \codim_{Y} \overline{f(X'')} \leq n =
\dim X'_{f(\eta)} - \codim_{Y'} \overline{f(X')}.$$
But applying the catenary hypothesis together with the hypothesis on
the universal relative dimension of $Y' \hookrightarrow Y$, we have
\begin{align*}
\codim_{Y'} \overline{f(X')} -\codim_Y \overline{f(X'')} 
&= \codim_Y \overline{f(X')}-\codim_Y Y' 
-\codim_Y \overline{f(X'')} \\
&  \leq \codim_Y \overline{f(X')}-\codim_{X''} X'-\codim_Y 
\overline{f(X'')}  \\
& = \codim_{\overline{f(X'')}} \overline{f(X')}-\codim_{X''} X'.\end{align*}
But the catenary hypothesis also allows us to use the dimension formula,
Proposition \ref{prop:dim-formula} 
to conclude that the righthand side above is equal 
to $\dim X'_{f(\eta)}-\dim X''_{f(\eta')}$, yielding the desired inequality.
\end{proof}

\section{Formal properties of relative dimension}\label{sec:rel-dim-comps}

In this section, we investigate the formal behavior of relative dimension,
including behavior with respect to compositions in general, and with
respect to composition with smooth morphisms.

Our first observation is the following, which will allow us to 
generalize Definition \ref{def:rel-dim} to morphisms of stacks.

\begin{prop}\label{prop:rel-dim-smooth-precomp} Suppose that $f:X \to Y$
is smooth, and $g:Y \to Z$ a morphism
with $Z$ universally catenary. Given $x \in X$, we have that $g$ has relative 
dimension at least $n$ at $f(x)$ if and only if $g \circ f$ has relative 
dimension at least $n+\reldim_x f$ at $x$. Additionally, 
the same statement holds for universal relative dimension.
\end{prop}

\begin{proof} The statements being local at $x$,
we may assume that $f$
has pure relative dimension $m:=\reldim_x f$. Since smooth morphisms are
open, we reduce immediately to showing that if $f$ is surjective, then
$g$ has (universal) relative dimension at least $n$ if and only if 
$g \circ f$ has (universal) relative dimension at least $n+m$. 
Because smoothness is preserved under base change, and 
composition commutes with base change, the universal statement follows
immediately from the non-universal statement.
Next, observe that if $g \circ f$
is locally of finite type 
then $g$ is locally of finite type 
by Lemma 17.7.5 of \cite{ega44}.

Now, if $Z'$ is an irreducible
closed subscheme of $Z$, then by smoothness, every component $X'$ of $X|_{Z'}$
dominates a component $Y'$ of $Y|_{Z'}$, and by surjectivity of $f$, 
every component $Y'$ of $Y|_{Z'}$ is dominated by a 
component $X'$ of $X|_{Z'}$.
Thus, given such $X'$ and $Y'$, having generic points $\eta$ and $\xi$
respectively, we have $g(f(\eta))=g(\xi)$, and
by smoothness $\dim X'_{g(f(\eta))} =\dim Y'_{g(\xi)}+ m$; the desired
statement follows.
\end{proof}

Now we come to the more substantial statement that Definition 
\ref{def:rel-dim} behaves well with respect to composition. 

\begin{lem}\label{lem:rel-dim-composition} Suppose that $f:X \to Y$ has
relative dimension at least $m$ at $x$, and $g:Y \to Z$ has relative 
dimension at least $n$ at $f(x)$. Then $g \circ f$ has relative dimension 
at least $m+n$ at $x$.

The same holds for universal relative dimension.
\end{lem}

\begin{proof} By restricting to suitable neighborhoods of $x$ and $f(x)$,
we may assume that $f$ and $g$ have (universal) relative dimension at least 
$m$ and $n$ everywhere, and we show that the composition has (universal)
relative at least $m+n$ everywhere. The universal statement follows 
immediately from the non-universal statement, since composition commutes 
with base change.

For the non-universal statement, let $Z'$ be an irreducible component of $Z$, 
and $X'$ an irreducible component of $X|_{Z'}$, with generic point $\eta$;
then we wish to show that 
$$\dim X'_{g(f(\eta))} - \codim_{Z'} \overline{g(f(X'))} \geq m+n.$$
Let $Y'$ be an irreducible component of $Y|_{Z'}$ containing 
$f(\eta)$, with generic point $\xi$.
$$\begin{tikzcd}
\eta \in X' \arrow[hook]{r}\arrow[xshift=1.8ex]{d} & X|_{Z'} \arrow[hook]{r}\arrow{d} 
& X \arrow{d}{f} \\
\xi \in Y' \arrow[hook]{r}\arrow{dr} & Y|_{Z'} \arrow[hook]{r}\arrow{d} & 
Y \arrow{d}{g} \\
& Z' \arrow[hook]{r} & Z
\end{tikzcd}$$
Then $X'$ is an irreducible component of $X|_{Y'}$, so we have
$$\dim X'_{f(\eta)} - \codim_{Y'} \overline{f(X')} \geq m, \quad \text{ and }
\quad \dim Y'_{g(\xi)} - \codim_{Z'} \overline{g(Y')} \geq n.$$
Now, if $\zeta$ is any generic point of an irreducible component $Y''$ of
the fiber $Y'_{g(f(\eta))}$ with $f(\eta)\in Y''$, note that 
$g(\zeta)=g(f(\eta))$.
$$\begin{tikzcd}
\quad\quad X'_{f(\eta)} \arrow[hook]{r}\arrow[xshift=1.5ex]{d} 
& X'_{g(f(\eta))} \arrow{d} \arrow{dl} \\
\zeta \in Y'' \arrow[hook]{r} & Y'_{g(f(\eta))}
\end{tikzcd}$$
The relationship between $\dim X'_{g(f(\eta))}$ and 
$\dim X'_{f(\eta)}$ is given by
$$\dim X'_{g(f(\eta))} = \dim X'_{f(\eta)}+\dim_{\zeta} Y'_{g(f(\eta))}-
\codim_{Y''} \overline{f(X')}.$$
Adding the two inequalities and using the above equation, we find that
\begin{multline*}\dim X'_{g(f(\eta))} -\dim_{\zeta} Y'_{g(f(\eta))}+ 
\codim_{Y''} \overline{f(X')}- \codim_{Y'} \overline{f(X')} \\
+ \dim Y'_{g(\xi)} - \codim_{Z'} \overline{g(Y')} \geq m+ n.
\end{multline*}

We now apply our catenary hypothesis, first to conclude that
$$\codim_{Y''} \overline{f(X')}=\codim_{Y'} \overline{f(X')}-
\codim_{Y'} Y''$$
and
$$\codim_{\overline{g(Y')}} \overline{g(f(X'))}
=\codim_{Z'} \overline{g(f(X'))}-\codim_{Z'} \overline{g(Y')}.$$
and second, to apply the dimension formula, Proposition 
\ref{prop:dim-formula}, to the morphism $Y' \to \overline{g(Y')}$ at the 
point $\zeta$, concluding that
$$\codim_{Y'} Y'' = \codim_{\overline{g(Y')}} \overline{g(Y'')} + 
\dim Y'_{g(\xi)} - \dim_{\zeta} Y'_{g(f(\eta))}.$$
Putting these three equations together with the previous inequality gives
the desired inequality.
\end{proof}

We conclude that Definition \ref{def:rel-dim} behaves well with respect
to smooth covers.

\begin{cor}\label{cor:rel-dim-smooth-local} Suppose that $f:X \to Y$ is a
morphism, with $Y$ universally catenary. Let $g:Y' \to Y$ be a smooth
cover, and let $h:X'' \to X'$ be a smooth cover,
where $X'=X \times_Y Y'$. Denote by $f''$ the
induced morphism $X'' \to Y'$. Then $f$ has universal relative dimension 
at least $n$ at a point $x$ if and only if $f''$ has universal relative 
dimension at least $n+\reldim_{x''} h$ at some (equivalently, at every) 
$x'' \in X''$ lying over $x$.
\end{cor}

We remark that without universality, relative dimension at least $n$ need
not be preserved under smooth base change; see Remark 
\ref{rem:smooth-comp-bad}.

\begin{proof} 
Denote by $f'$ the morphism $X' \to Y'$, so that 
$f''=f' \circ h$. 
$$\begin{tikzcd}
X'' \arrow{r}{h} \arrow{dr}{f''} & X' \arrow{r}{g'} \arrow{d}{f'} 
& X \arrow{d}{f} \\
{} & Y' \arrow{r}{g} & Y
\end{tikzcd}$$
If $f$ has universal relative dimension at least $n$ at
$x$, then by definition $f'$ has universal relative dimension at $x'$
for any $x'$ lying over $x$, and in particular for $x'=h(x'')$. We 
conclude that $f''$ has
universal relative dimension at least $n+\reldim_{x''} h$ from Proposition 
\ref{prop:rel-dim-smooth-precomp}. Conversely, if $f''$ has universal 
relative dimension at least $n+\reldim_{x''} h$ at $x''$, then the other 
direction of Proposition \ref{prop:rel-dim-smooth-precomp} implies that 
$f'$ has universal relative dimension at least $n$ at $h(x'')$. Then by 
Lemma \ref{lem:rel-dim-composition}, we have that
$g \circ f'$ has universal relative dimension at least 
$n+\reldim_{f'(h(x''))} g$ at $h(x'')$.
Write $g':X' \to X$ for the base change of $g$; then $g'$ is a smooth
cover with $\reldim_{h(x'')} g'=\reldim_{f'(h(x''))} g$, and 
$g \circ f'=f \circ g'$, so again
applying Proposition \ref{prop:rel-dim-smooth-precomp}, we have that
$f$ has universal relative dimension at least $n$ at $g'(h(x''))=x$, as 
desired.
\end{proof}

Using Lemma \ref{lem:rel-dim-composition}
and the standard Grothendieck six conditions argument (see for instance
Remark 5.5.12 of \cite{ega1}),
we see that Definition \ref{def:rel-dim} also behaves well under
post-composition by a smooth morphism.

\begin{cor}\label{cor:rel-dim-smooth-postcomp}
Given morphisms $f:X \to Y$ and $g:Y \to Z$, suppose $g$ is 
smooth of relative dimension $n$, with $Z$ universally catenary. Then $f$ 
has universal relative dimension at least $m$ if and only if $g \circ f$ has 
universal relative dimension at least $m+n$.
\end{cor}

\begin{proof} The ``only if'' direction is immediate from Lemma
\ref{lem:rel-dim-composition}. For the converse,
first observe that $Y$ is universally catenary, and 
the smoothness of $g$ implies the diagonal $\Delta_g: Y \to Y \times_Z Y$
has universal relative dimension at least $-n$, because if we factor 
$\Delta_g$ as a closed immersion followed by an open immersion, the closed 
immersion is locally cut out by $n$ equations (See Proposition 2.2.7
of \cite{b-l-r}).
Now, if $g \circ f$ has universal relative dimension at 
least $m+n$, then the second projection $p_2:X \times _Z Y \to Y$ is the base 
change of $g \circ f$ to $Y$, so likewise has universal relative dimension
at least $m+n$. On the other hand, the graph morphism 
$\Gamma_f:X \to X \times_Z Y$ is the base change of $\Delta_g$ under 
$f \times \id$, so it has universal relative dimension at least $-n$. 
Since $f=p_2 \circ \Gamma_f$, we conclude that $f$ has universal relative
dimension at least $m$ from Lemma \ref{lem:rel-dim-composition}.
\end{proof}

\section{Examples and further discussion}\label{sec:rel-dim-exs}

We conclude our examination of relative dimension with positive and
negative examples, and a number of remarks on various aspects of the
subject.

First, as another consequence of Lemma \ref{lem:rel-dim-composition},
we obtain the following wide classes of examples.

\begin{cor}\label{cor:rel-dim-examples} Suppose that $f:X \to Y$ is a 
closed immersion, and $g:Y \to Z$ is smooth of relative dimension $n$,
with $Z$ universally catenary. If either $Z$ is regular and every component
of $X$ has codimension at most $c$ in $Y$, or $X$ may be expressed locally
as an intersection of determinantal conditions with expected codimensions
adding up to $c$, then $g\circ f$ has universal relative dimension at least 
$n-c$.

Alternatively, suppose that $f:X \to Y$ is a morphism of smooth $S$-schemes,
with $S$ universally catenary, and $m$ and $n$ the relative dimensions of
$X$ and $Y$ over $S$, respectively. Then $f$ has universal relative 
dimension at least $m-n$.
\end{cor}

Again, we mention that universal relative dimension is insensitive to 
non-reduced structures, so in fact the descriptions of the corollary
only need to hold up to taking reduced structures in order to conclude
the desired statements.

\begin{proof} The first statement is a direct consequence of Lemma
\ref{lem:rel-dim-composition}, together with Examples \ref{ex:flat} and 
\ref{ex:closed-imm}. For the second, we observe that $f$ is necessarily
locally of finite type, so $X$ can locally 
be written as a closed subscheme of $\AA^N_Y$. Then, because $X$ is 
smooth over $S$ of relative dimension $m$, and $\AA^N_Y$ is smooth over
$S$ of relative dimension $N+n$, we have that $X$ is everywhere locally
cut out by $N+n-m$ equations inside $\AA^N_Y$ by Proposition 2.2.7
of \cite{b-l-r}, so we have reduced the second statement to the first.
\end{proof}

\begin{ex}\label{ex:blowups} 
If $Y$ is a smooth variety over a field, and $f:X \to Y$ is a blowup with
smooth center, the second part of Corollary \ref{cor:rel-dim-examples}
implies that $f$ has strong relative dimension at least $0$.
\end{ex}

We next provide some negative examples.

\begin{ex}\label{ex:normalization} As usual, the normalization 
$\widetilde{C}$ of an irreducible nodal curve $C$ provides an interesting 
example to consider. It has relative
dimension at least $0$, but not universal relative dimension at least
$0$. Indeed, if we take $\widetilde{C} \times_C \widetilde{C}$, we obtain 
$\widetilde{C}$ together with two isolated points, each of which maps to one
of the preimages of the node under projection to $\widetilde{C}$.

Another base change which does not have relative dimension at least $0$
is obtained by taking the product with $\widetilde{C}$, considered over
the base field. In this case, if we let 
$\Delta \subseteq \widetilde{C} \times \widetilde{C}$ be the diagonal, and
$Z \subseteq C \times \widetilde{C}$ its image, then the restriction of 
$\widetilde{C} \times \widetilde{C}$ to (the preimage of) $Z$ likewise 
consists of a copy of $\Delta\cong \widetilde{C}$ together with two isolated 
points.

Notice that this means that relative dimension does not have a 
semicontinuity property: the map 
$\widetilde{C} \times \widetilde{C} \to C \times \widetilde{C}$ has relative 
dimension at least $0$ over the nonsingular locus of $C \times \widetilde{C}$, 
but not over all of $C \times \widetilde{C}$.
\end{ex}

\begin{ex}\label{ex:quadric-cone} An example of a closed immersion of
codimension $c$ which does not have relative dimension at least $-c$ is
given by the standard example of failure of subadditivity of codimension
for intersections: let $X$ be a cone over a quadric surface, and $Z$
the cone over a line in the surface. Then $Z$ has codimension $1$, but
we claim that the inclusion $Z \to X$ does not have relative dimension
at least $-1$. Indeed, if $Z'$ is the cone over any other line in the
same ruling, then $Z' \cap Z$ is equal to the cone point, which has
codimension $2$ in $Z'$.
\end{ex}

\begin{rem}\label{rem:map-dependence}
According to Corollary \ref{cor:rel-dim-examples}, relative dimension
does not depend on the map $f$ when $X$ and $Y$ are both smooth over a
common base. However, in the general case, we see from Example 
\ref{ex:normalization} that there is a nontrivial dependence on $f$.
Indeed, the morphism 
$\widetilde{C} \times \widetilde{C} \to C \times \widetilde{C}$ 
considered in that example does not have relative dimension at least $0$,
but there are many other morphisms with the same source and target which
do. For instance, if we choose a point $P \in \widetilde{C}$ and consider
the composition
$$\widetilde{C} \times \widetilde{C} \overset{p_1}{\to} \widetilde{C} \to C
\to C \times \{P\} \subseteq C \times \widetilde{C},$$
where the second morphism is the normalization, then we have a composition
of morphisms of relative dimension at least $1$, $0$ and $-1$ respectively,
so obtain relative dimension at least $0$, as asserted.
\end{rem}

\begin{rem}\label{rem:smoothing} We elaborate slightly on the hypotheses
and conclusions of Proposition \ref{prop:rel-dim-smoothing-2}. First,
we observe that under the hypotheses of the proposition, although
$\dim X'_{f(\eta)} - \codim_{Y'} \overline{f(X')}$ remains unchanged when
we replace $X'$ by $X''$, the separate terms need not. Indeed, the map 
$\AA^2 \to \AA^2$ given by $(x,y) \mapsto (x,xy)$
has dimension dimension at least $0$, and if we restrict to the line
$x=0$ in the target, we get that the 
$\dim X'_{f(\eta)} = \codim_{Y'} \overline{f(X')}=1$, so the proposition 
applies. However, without restricting (and letting $\eta'$ be the generic
point of $X$), we have
$\dim X_{f(\eta')} = \codim_{Y} \overline{f(X)}=0$. This behavior is discussed
further in Remark \ref{rem:constructible}.

The final conclusion of Proposition \ref{prop:rel-dim-smoothing-2} may
seem superfluous in the context of smoothing arguments, but in fact it
plays an important role in inductive arguments.
We see from the second part of Example \ref{ex:normalization} that for this 
final conclusion, hypothesis (III) is indeed necessary. 
In the notation of that example, consider the map
$$\widetilde{C} \times \widetilde{C} \smallsetminus \Delta 
\to C \times \widetilde{C}.$$
This has relative dimension at least $-1$,
and if we restrict to $Z$, we find that the conditions
of Proposition \ref{prop:rel-dim-smoothing-2} are satisfied with $n=-1$, 
except that $Z$ does not have relative dimension at least $-1$ in 
$C \times \widetilde{C}$.
However, in this case (again setting $\eta'$ to be the generic point of $X$) 
we have $\dim X_{f(\eta)} - \codim_{Y} \overline{f(X)}=0$.
\end{rem}

\begin{rem}\label{rem:constructible}
Relative dimension of morphisms can also shed light on the behavior
of constructibility of images, in the sense of understanding when a 
dominant morphism fails to have open image. The first example of this
is always the map $\AA^2 \to \AA^2$ given by $(x,y) \mapsto (x,xy)$,
and it is natural to notice that the failure of openness occurs at a
jump in fiber dimension, and to wonder whether this phenomenon is 
general. 

In fact, Example \ref{ex:normalization} yields
an example of a dominant morphism of varieties for which the image is not 
open and the fiber dimension does not jump. Namely, in the notation of
that example, consider the map
$$\widetilde{C} \times \widetilde{C} \smallsetminus \Delta \to 
C \times \widetilde{C}.$$
This is a dominant morphism of varieties, with all fibers $0$-dimensional,
but its image is $(C \times \widetilde{C} \smallsetminus Z) \cup \{P_1,P_2\}$,
where $P_1$ and $P_2$ are the intersection of $Z$ with the singular locus
of $C \times \widetilde{C}$. 

However, we see immediately from Proposition \ref{prop:rel-dim-smoothing}
that if $f:X \to Y$ has relative dimension at least $n$, and there is
no jumping of fiber dimension (i.e., all fibers have dimension $n$), then
$f$ necessarily has open image. By Corollary \ref{cor:rel-dim-examples},
the relative dimension hypothesis applies in particular whenever $X$ and
$Y$ are smooth and $\dim X-\dim Y=n$. Thus, the behavior of the first
example is in fact rather general.
\end{rem}

\begin{rem}\label{rem:compose-w-immersion} Suppose that we have
$f:X \to Y$ and $g:Y \to Z$, where $g$ is a closed immersion, and
$g \circ f$ has relative dimension at least $n$. Then it is immediate
from the definition that $f$ has relative dimension at least $n$ as well,
but it is natural to wonder, if $g(Y)$ has codimension $c$ in $Z$, 
whether in fact $f$ has relative dimension at least $n+c$. However, this
is not the case. For instance, if $f$ is the inclusion of the cone over
a line into the cone over a quadric surface, as in Example 
\ref{ex:quadric-cone}, then we know that $f$ does not have relative 
dimension at least $-1$. However, if we compose with the inclusion of
$Y$ into $\AA^3$, then because $\AA^3$ is regular, we have that the 
composition has relative dimension at least $-2$, in fact universally.
\end{rem}

\begin{rem}\label{rem:smooth-comp-bad} Example \ref{ex:normalization}
also demonstrates that the statement of Corollary
\ref{cor:rel-dim-smooth-postcomp} fails if we replace universal relative
dimension with relative dimension. Indeed, in the notation of the
example, the composed morphism 
$$\widetilde{C} \times \widetilde{C} \to C \times \widetilde{C} \to C$$
has relative dimension at least $1$, and $C \times \widetilde{C} \to C$ is
smooth of relative dimension $1$, but
$\widetilde{C} \times \widetilde{C} \to C \times \widetilde{C}$ does not 
have relative dimension at least $0$.
\end{rem}

\section{Codimension for stacks}\label{sec:stack-codim}

We now give a treatment of the theory of codimension for algebraic
stacks. Although the definitions and statements are as expected, there
are some non-trivial aspects to the proofs, largely coming from the fact 
that basic statements on codimension for schemes require irreducibility, 
and irreducibility is not preserved under smooth covers. 

The definition is as one would expect.

\begin{defn} Let $\cX$ be a locally Noetherian
algebraic stack, and $\cZ \subseteq \cX$ a
closed substack. The {\bf codimension} of $\cZ$ in $\cX$ is defined to be
the codimension in $X$ of $X \times_{\cX} \cZ$, where $X \to \cX$ is a
smooth presentation of $\cX$.
\end{defn}

We remark that although the Zariski topology on $\cX$ is sober, 
topological codimension does not yield the correct definition -- see
Example \ref{ex:stack-codim}.

\begin{prop}\label{prop:codim-defd} Codimension is well defined, and finite.
\end{prop}

\begin{proof} We first verify that codimension is well defined. If 
$X \to \cX$, $X' \to \cX$
are both smooth presentations of $\cX$, then
$X \times_{\cX} X'$ is a smooth cover of both $X$ and $X'$; in principle,
it may be that $X \times_{\cX} X'$ is an algebraic space, but even so we
can take a further smooth cover to obtain a common smooth cover of $X$ and
$X'$ be a scheme. Thus, it is
enough to see that codimension is preserved under passage to a smooth
cover. But this is true more generally for flat covers, see
Corollary 6.1.4 of \cite{ega42}.

We next verify finiteness. By definition, any smooth presentation
$P:X \to \cX$ has
$X$ locally Noetherian. Now, $\codim_{\cX} \cZ := \codim_X P^{-1} Z$,
and in order to establish finiteness of $\codim_X P^{-1} Z$, it is enough
to show that $\codim_X Z'$ is finite for some irreducible component $Z'$
of $P^{-1} Z$. But let $Z'$ be any irreducible component; since $X$ is
locally Noetherian, there exists an open subscheme $U \subseteq X$ which
meets $Z'$ and is Noetherian. Then $\codim_X Z'= \codim_U Z'$, and
$\codim_U Z'$ is finite, as desired.
\end{proof}

\begin{prop}\label{prop:stack-codim-comp} Let $\cZ \subseteq \cX$ be a
closed substack. Then $\codim_{\cX} \cZ=0$ if and only if $\cZ$ contains
an irreducible component of $\cX$. 
\end{prop}

\begin{proof} The statement is obvious for schemes. Thus, if $P:X \to \cX$
is a smooth presentation, and $Z=P^{-1}(\cZ)$ it suffices to prove that 
$\cZ$ contains an irreducible component of $\cX$ if and only if $Z$ 
contains an irreducible component of $X$. But this is clear from the
fact that $P$ maps generic points to generic points.
\end{proof}

We now state our main intrinsically stack-theoretic results, deferring for
the moment the supporting results needed to prove them. We begin with 
a statement to the effect that, in the case of stacks, codimension still 
behaves the same with respect to irreducible components. 

\begin{prop}\label{prop:stack-codim} Let
$\cZ \subseteq \cX$ be a closed substack. Then
\begin{align*} \codim_{\cX} \cZ 
& = \min_{\cZ' \subseteq \cZ} \codim_{\cX} \cZ' \\
& = \min_{\cZ' \subseteq \cZ} 
\left(\max_{\cX' \subseteq \cX: \cZ' \subseteq \cX'}
\codim_{\cX'} \cZ'\right),
\end{align*}
where $\cZ'$ and $\cX'$ range over irreducible components (or equivalently,
irreducible closed substacks) of $\cZ$ and $\cX$ respectively.
\end{prop}

The next two results are generalizations to stacks of standard statements
in the scheme setting, both of which require catenary hypotheses. The 
first is additivity of codimension.

\begin{prop}\label{prop:stack-codim-add}
Let $\cX$ be an irreducible algebraic stack,
locally of finite type over a universally catenary base scheme, and
$\cZ' \subseteq \cZ \subseteq \cX$ irreducible closed substacks.
Then
$$\codim_{\cX} \cZ'=\codim_{\cX} \cZ+\codim_{\cZ} \cZ'.$$
\end{prop}

The second is the generalization of the ``dimension formula''
to stacks.

\begin{prop}\label{prop:dim-formula-stacks} Let $\cY$ be an algebraic
stack, of finite type over a universally catenary scheme, 
and $f:\cX \to \cY$ a morphism locally of finite
type. Let $\cZ' \subseteq \cZ$ be irreducible closed subschemes of $\cX$, 
and let $\eta'$ and $\eta$ be their respective
generic points. Then
$$\codim_{\cZ} \cZ'-\codim_{\overline{f(\cZ)}} \overline{f(\cZ')}
= \dim \cZ_{f(\eta)}- \dim \cZ'_{f(\eta')}.$$
\end{prop}

Propositions \ref{prop:stack-codim}, \ref{prop:stack-codim-add} and 
\ref{prop:dim-formula-stacks} will be easy consequences of the following.

\begin{lem}\label{lem:irred-codim} Let $\cX$ be a locally Noetherian 
irreducible algebraic stack,
and $\cZ \subseteq \cX$ an irreducible closed substack. Let $P:X \to \cX$
be a smooth presentation of $\cX$, and $P':Z \to \cZ$ the induced 
presentation of $\cZ$. Then for every irreducible component $Z'$ of $Z$,
we have
$$\codim_{\cX} \cZ=\codim_{X} Z'.$$
If further $\cX$ is locally of finite type over a universally catenary base 
scheme, then for every irreducible component $Z'$ of $Z$,
and every irreducible component $X'$ of $X$ containing $Z'$, we have
$$\codim_{\cX} \cZ=\codim_{X'} Z'.$$
\end{lem}

As usual, the proof of the lemma reduces to a related result on
schemes, which we state separately.

\begin{prop}\label{prop:irred-codim} Let $f:X \to Y$ be a smooth morphism
of schemes, with $Y$ irreducible and locally Noetherian. Let 
$Z \subseteq Y$ be an irreducible closed subscheme. Then if $Z'$ is any
irreducible component of $f^{-1}(Z)$, we have
$$\codim_Y Z = \codim_{X} Z'.$$
If further $Y$ is universally catenary, then for any irreducible component
$Z'$ of $f^{-1}(Z)$, and any irreducible component $X'$
of $X$ containing $Z'$, we have
$$\codim_Y Z = \codim_{X'} Z'.$$
\end{prop}

\begin{proof} Using openness of smooth morphisms, we may replace
$X$ and $Y$ by open subsets on which $f$ is surjective, and $f^{-1}(Z)=Z'$.
The first statement then follows from Corollary 6.1.4 of \cite{ega42}.

In the universally catenary case,
since the relative dimension of a smooth morphism is locally constant,
and smoothness also ensures that $X'$ dominates $Y$ and $Z'$ dominates
$Z$, this is an immediate application of the dimension formula,
Proposition \ref{prop:dim-formula}.
\end{proof}

\begin{proof}[Proof of Lemma \ref{lem:irred-codim}] 
Our first claim is that if $X'' \supseteq Z''$ are different 
choices of irreducible components of $X$ and $Z$ respectively, then
exists a scheme $Y$ with irreducible closed subschemes 
$\widetilde{Z}' \subseteq \widetilde{X}'$ such that $Y$
is a smooth cover of $X$ in two different ways, and under these maps,
$\widetilde{Z}'$ is an irreducible component of the preimage of $Z'$
and $Z''$ respectively, and $\widetilde{X}'$ is an irreducible component
of the preimage of $X'$ and $X''$ respectively. Indeed,
by considering induced mappings on generic points, we can find an
irreducible component $\widetilde{Z}$ of $Z \times_{\cX} Z$ which
dominates $Z'$ and $Z''$ respectively under the projection morphisms,
as well as an irreducible component $\widetilde{X}$ of
$X \times_{\cX} X$ which contains $\widetilde{Z}$ and likewise 
dominates $X'$ and $X''$ respectively under the projection morphisms.
Note also that
$$Z \times_{\cX} Z=Z \times_{\cX} X = Z \times_{X} (X \times_{\cX} X)
=X \times_{\cX} Z = (X\times_{\cX}X) \times_X Z,$$
so we have also that $\widetilde{Z}$ is simultaneously an irreducible
component of the preimage in $X \times_{\cX} X$ of $Z'$ via one projection,
and of $Z''$ via the other. Now, it may be that $X \times_{\cX} X$ is 
an algebraic space rather than a scheme, but if so we can pass to a further 
smooth cover $Y \to X \times_{\cX} X$ and let $\widetilde{X}'$ and 
$\widetilde{Z}'$ be components of the preimages of $\widetilde{X}$ and 
$\widetilde{Z}$. This proves the desired claim.

Now, for the first statement, it suffices to show that 
$$\codim_{X} Z' = \codim_{X} Z''.$$
Choosing $Y$ and $\widetilde{Z}'$ as above (with the choice of $X''$ and 
$\widetilde{X}'$ being irrelevant), we apply the first part of 
Proposition \ref{prop:irred-codim} to conclude that
$$\codim_{X} Z' = \codim_{Y} \widetilde{Z}'=\codim_{X} Z''.$$

For the second statement, it suffices to show that 
$$\codim_{X'} Z' = \codim_{X''} Z''.$$
Choosing $\widetilde{X}'$ and $\widetilde{Z}'$ as above,
we apply the second part of 
Proposition \ref{prop:irred-codim} to deduce that
$$\codim_{X'} Z' = \codim_{\widetilde{X}'} \widetilde{Z}' = \codim_{X''} Z'',$$
as desired.
\end{proof}

We now give the proofs of our main propositions.

\begin{proof}[Proof of Proposition \ref{prop:stack-codim}]
The first equality is an immediate consequence of the definitions of
codimension for subschemes and substacks. 
The second equality reduces immediately to the
case that $\cZ=\cZ'$, which is to say that we need to verify that,
when $\cZ$ is irreducible, we have
$$\codim_{\cX} \cZ = \max_{\cX' \subseteq \cX:\cZ \subseteq \cX'}
\codim_{\cX'} \cZ.$$
Letting $P:X \to \cX$ be a smooth presentation, the above equation gives
us that
$$\min_{Z \subseteq P^{-1}(\cZ)} \codim_X Z= 
\max_{\cX' \subseteq \cX:\cZ \subseteq \cX'} \min_{Z \subseteq
P^{-1}(\cZ)} \codim_{X'} Z,$$
where $X'=P^{-1}(\cX')$, and $Z$ ranges over irreducible components of
$P^{-1}(\cZ)$. However, according to Lemma \ref{lem:irred-codim}, the
values of $\codim_{X} Z$ and $\codim_{X'} Z$ are each independent of the
choice of $Z$, so the desired identity reduces to the standard identity
$$\codim_{X} Z=\max_{X' \subseteq X: Z \subseteq X'} \codim_{X'} Z.$$
\end{proof}

\begin{proof}[Proof of Proposition \ref{prop:stack-codim-add}]
The issue here is that even when everything is of finite type over
a field, the statement can fail without irreducibility hypotheses, and 
since irreducibility is not preserved under smooth covers, we again have
to exercise care in reducing to the scheme case. However, according to
Lemma \ref{lem:irred-codim} we can compute all the codimensions in question
in terms of a fixed chain of irreducible closed subschemes in a smooth
presentation of $\cX$, so the statement reduces to the usual statement for 
schemes, which is immediate from the definition of catenary.
\end{proof}

\begin{proof}[Proof of Proposition \ref{prop:dim-formula-stacks}] Let 
$P: Y\to\cY$ and $P':X \to \cX \times _{\cY} Y$ be a smooth presentation 
of $f$, with $f':X \to Y$ the induced morphism. By hypothesis, $f'$ is
locally of finite type, and $Y$ is universally catenary.
Let $Z$ and $Z'$ be the preimages in $X$ of $\cZ$ and $\cZ'$
respectively, and let $W'$ be an irreducible component of $Z'$,
and $W$ an irreducible component
of $Z$ containing $W'$. Then $\overline{f'(W)}$ and $\overline{f'(W')}$
are irreducible components of $P^{-1}(\overline{f(\cZ)})$ and
$P^{-1}(\overline{f(\cZ')})$, respectively.
It follows from Lemma \ref{lem:irred-codim} that we have
$$\codim_{\cZ} \cZ'=\codim_{W} W'\quad \text{ and } \quad 
\codim_{\overline{f(\cZ)}} \overline{f(\cZ')}=
\codim_{\overline{f'(W)}} \overline{f'(W')}.$$
On the other hand, the relative dimension of $P'$, being locally
constant, is the same at all points of $W$ and in particular on $W'$,
so we see that
$$\dim \cZ_{f(\eta)}- \dim \cZ'_{f(\eta')}
=\dim W_{f'(\xi)}-\dim W'_{f'(\xi')},$$
where $\xi$ and $\xi'$ are the generic points of $W$ and $W'$ respectively.
The proposition then follows from Proposition \ref{prop:dim-formula}.
\end{proof}

We conclude with examples illustrating the obstruction to defining a
notion of ``dimension of local ring'' for stacks, and the difference
between codimension and topological codimension for Artin stacks.

\begin{ex}\label{ex:a1-mod-gm} Let $k$ be a field, and $\cX$ the stack
given by $\AA^1_k/\Gm$. Then $\cX$ is smooth of dimension $0$ over
$\Spec k$. It has two points $x_1$ and $x_0$, with $x_1$ specializing to 
$x_0$, and via the smooth presentation $\AA^1_k \to \cX$, we see that 
$x_0$ has codimension $1$ in $\cX$. This example demonstrates that there
is no theory of ``dimension of local rings'' for algebraic stacks -- even
those of finite type over a field -- which simultaneously satisfies the
following three conditions:

\begin{ilist}
\itm If $\delta(x)$ denotes the ``dimension of the local ring at $x$,''
and $x_1$ specializes to $x_0$, then 
$$\delta(x_0)=\delta(x_1)+\codim_{\overline{x}_1}\overline{x}_0.$$
\itm If $\cX$ is smooth over a field, of pure dimension $n$, and $x$ is
a closed point of $\cX$, then $\delta(x)=n$.
\itm $\delta(x)$ can be computed on any open neighborhood of $x$.
\end{ilist}

Indeed, conditions (ii) and (iii) together imply that in our example,
we must have $\delta(x_0)=0=\delta(x_1)$, but this contradicts (i).
\end{ex}

\begin{ex}\label{ex:stack-codim} Identify $\AA^4_k$ with $2 \times 2$
matrices over $k$, and let $\GL_2(k)$ act by left multiplication. The
orbits of this action are in bijection with the kernels of the associated
linear map. Thus, if $\cX$ is the quotient stack $[\AA^4_k/\GL_4(k)]$,
we see that the point corresponding to the zero matrix has codimension
$2$ in the underlying topological space of $\cX$, but codimension $4$
in $\cX$ itself.
\end{ex}

\section{Relative dimension and stacks}\label{sec:rel-dim-stacks}

Finally, we generalize our results on relative dimension of morphisms 
to the setting of algebraic stacks. 
First, following the standard procedures, we will obtain from
Corollary \ref{cor:rel-dim-smooth-local} that our definition of relative
dimension generalizes to stacks.

\begin{defn}\label{def:rel-dim-stacks} Suppose that $\cX,\cY$ are algebraic
stacks, locally of finite type over a universally catenary base scheme
$S$, and $f:\cX \to \cY$ a morphism. We say that $f$ has
\textbf{universal relative dimension at least $n$} if there exists a
smooth presentation $P:Y \to \cY$ and $P':X \to \cX \times_{\cY} Y$ of 
$f$ such that for all $x \in X$, the induced morphism $X \to Y$ has 
universal relative dimension at least $n+\reldim_x P'$ at $x$.
\end{defn}

The following proposition says in essence that universal relative 
dimension is well defined for stacks.

\begin{prop}\label{prop:rel-dim-stacks-defd} Suppose that $\cX,\cY$ are 
algebraic
stacks, locally of finite type over a universally catenary base scheme
$S$, and $f:\cX \to \cY$ a morphism of universal relative dimension at least
$n$. Then for all
smooth presentations $P:Y \to \cY$ and $P':X \to \cX \times_{\cY} Y$, we 
have that for all $x \in X$, the induced morphism $X \to Y$ has universal 
relative dimension at least $n+\reldim_x P'$ at $x$.
\end{prop}

The proof of this proposition is completely standard, and underlies
the fundamental Definition 4.14 of \cite{l-m-b}. However, since the
proof is omitted in \textit{loc. cit.}, we sketch it here for the
convenience of the reader.

\begin{proof} Suppose that $P$ and $P'$ satisfy the condition of
Definition \ref{def:rel-dim-stacks}, and we are given $Y' \to \cY$ and 
$X' \to \cX \times_{\cY} Y'$ yielding another smooth presentation of $f$.
Let $Y'' \to Y \times_{\cY} Y'$  and 
$X'' \to (X \times_{\cX} X') \times_{Y \times_{\cY} Y'} Y''$
be smooth presentations of the relevant algebraic spaces. Then the
morphism $X'' \to Y''$ factors through both 
$X \times_Y Y'' = (X \times_{\cY} Y') \times_{Y \times_{\cY} Y'} Y''$
and
$X' \times_{Y'} Y'' = (Y \times_{\cY} X') \times_{Y \times_{\cY} Y'} Y''$.
$$\begin{tikzcd}[column sep=small]
{} & X'' \arrow{d} & {} \\
{} & (X \times_{\cX} X')_{Y \times_{\cY} Y'} Y'' \arrow{dr}\arrow{dl} & {} \\
X \times_Y Y'' \arrow{dr} & {} & X' \times_{Y'} Y'' \arrow{dl} \\
{} & Y'' & {}
\end{tikzcd}$$
Moreover, because $X \times_{\cY} Y' = X \times_{\cX} (\cX \times_{\cY} Y')$,
we see that in fact $X''$ is a smooth cover of $X \times_Y Y''$,
and similarly of $X' \times_{Y'} Y''$.
Applying Corollary \ref{cor:rel-dim-smooth-local} twice, we conclude
that the hypothesis on the universal relative dimension of $X \to Y$
implies the desired statement on $X' \to Y'$.
\end{proof}

We next use the properties of stack codimension which we have developed
to verify that universal relative dimension for stacks could be defined 
without reference to smooth presentations.

\begin{prop}\label{prop:rel-dim-stacks-agree} Suppose that $\cX,\cY$ are 
algebraic stacks, locally of finite type over a universally catenary base 
scheme $S$, and $f:\cX \to \cY$ a morphism. Then $f$ has universal 
relative dimension at least $n$ if and only if for every irreducible
algebraic stack $\cZ$ over $\cY$, locally of finite type over a universally 
catenary scheme (possibly distinct from $S$), and every irreducible component 
$\cZ'$ of $\cX \times_{\cY} \cZ$, we have
$$\dim \cZ'_{f'(\eta)} - \codim_{\cZ} \overline{f'(\cZ')} \geq n,$$
where $\eta$ is the generic point of $\cZ'$, and 
$f':\cX \times_{\cY} \cZ \to \cZ$ is the second projection.
\end{prop}

\begin{proof} Fix $P:Y \to \cY$ and $P':X \to \cX \times_{\cY} Y$ a
smooth presentation of $f$.

First suppose that $f$ has universal
relative dimension at least $n$, and let $\cZ \to \cY$ be as in the
statement of the proposition. Fix a smooth cover $Z \to \cZ \times_{\cY} Y$,
and let $g:X \times_Y Z \to Z$ be the second projection.
We observe that $X \times_Y Z$ is a smooth cover of $\cX \times_{\cY} \cZ$:
indeed, we obtain this as the composition
$$X \times_Y Z \to (\cX \times_{\cY} Y) \times_Y Z = \cX \times_{\cY} Z
\to \cX \times_{\cY} \cZ.$$
Accordingly, we can let $Z' \subseteq X \times_Y Z$ be an irreducible
component of the preimage of $\cZ' \subseteq \cX \times_{\cY} \cZ$, and
$Z_0$ an irreducible component of $Z$ containing the image of $g(Z')$. 
$$\begin{tikzcd} 
Z' \arrow{rr} \arrow[hook]{dr} \arrow{dd} & {} & \cZ'\arrow[hook]{d} \\
{}& X \times_Y Z \arrow{r} \arrow{d}{g} & \cX \times_{\cY} \cZ 
\arrow{d}{f'} \\
Z_0 \arrow[hook]{r} &  Z \arrow{r} & \cZ
\end{tikzcd}$$
Then $Z'$ is an irreducible component of $X \times_Y Z$, and hence also 
of the preimage of $Z_0$ in $X \times_Y Z$. Let 
$\eta'$ be the generic point of $Z'$,
and $\xi$ its image in $X$. Then by hypothesis, we have
$$\dim Z'_{g(\eta')}-\codim_{Z_0} \overline{g(Z')} \geq n+\reldim_{\xi} P'.$$

We claim that we have
$$\dim Z'_{g(\eta')}=\dim \cZ'_{f'(\eta)}+\reldim_{\xi} P', \text{ and }
\codim_{Z_0} \overline{g(Z')} = \codim_{\cZ} \overline{f'(\cZ')};$$
together, these claims yield the desired inequality. For the latter claim,
according to Lemma \ref{lem:irred-codim}, it is enough to show that
$\overline{g(Z')}$ is an irreducible component of the preimage of
$\overline{f'(\cZ')}$. But because $X \times_Y Z \to \cX \times_{\cY} \cZ$
is smooth, we have that $Z'$ dominates $\cZ'$, so this is clear.
Similarly, if we let $\cZ''$ be an irreducible component of 
$\cX \times_{\cY} Z$ contained in the preimage of $\cZ'$ and containing the
image of $Z'$, then we see that $\dim \cZ'_{f'(\eta)}=\dim \cZ''_{g(\eta')}$
since the latter is (an irreducible component of) a base extension of the 
former, and $\dim Z'_{g(\eta')}=\dim \cZ''_{g(\eta')}+\reldim_{\xi} P'$,
since $Z'$ is (an irreducible component of) a smooth cover of $\cZ''$,
obtained as a base change of $P'$.
We thus conclude one direction of the proposition.

For the converse, given $x \in X$, restricting to the connected component
of $X$ containing $x$, we may assume that $X \to \cX \times_{\cY} Y$ has
constant relative dimension $m$, and we want to prove that $X \to Y$ has
universal relative dimension $n+m$. Suppose we have $Z \to Y$, with $Z$
irreducible and universally catenary, and let $Z'$ be an irreducible
component of $X \times_Y Z$, with generic point $\eta$; if 
$g:X \times_Y Z \to Z$ is the second projection, we wish to show that
$$\dim Z'_{g(\eta)}-\codim_{Z} \overline{g(Z')} \geq n+m.$$
Now, let $\cZ'$ be an irreducible component of $\cX \times_{\cY} Z$
containing the image of $Z'$ under the induced map 
$X\times_Y Z \to \cX \times_{\cY} \cZ$, and let $\eta'$ be the generic
point of $\cZ'$. By hypothesis, we have
$$\dim \cZ'_{g'(\eta')}-\codim_{Z} \overline{g'(\cZ')} \geq n,$$
where $g':\cX \times_{\cY} Z \to Z$ is the second projection.
We now observe that 
$$X \times_Y Z=X \times_{\cX \times_{\cY} Y} \cX \times_{\cY} Z,$$
so $Z'$ is an irreducible component of a scheme smooth of relative 
dimension $m$ over $\cZ'$, and in particular also dominates $\cZ'$.
The desired inequality follows.
\end{proof}

We can now state the main application of relative dimension of morphisms
in the context of stacks. Recall that we have defined the notions of
reducedness of fibers and of the fiber dimension function $\delta_x f$
in Definitions \ref{def:stack-reduced-fibers} and \ref{def:fiber-dim-stacks}.

\begin{cor}\label{cor:rel-dim-smoothing-stacks} 
If $f:\cX \to \cY$ has universal relative dimension
at least $n$, then $\delta_x f \geq n$ for all $x \in \cX$.
If for some $x \in \cX$, we have $\delta_x f=n$,
then there exists a neighborhood $\cU$ of $x$ with $\delta_{x'} f=n$ for
all $x' \in \cU$, and on any such neighborhood we have that $f$ is
universally open.

If further $\cY$ is reduced and $f$ is reduced
at $x$, then $f$ is flat at $x$.
\end{cor}

\begin{proof} It is immediate from the definitions and from Proposition
\ref{prop:rel-dim-prelim} that $\delta_x f \geq n$ for all $x \in \cX$,
and it then follows from Proposition \ref{prop:fiber-dim-stacks} that if 
$\delta_x f =n$, there exists an open neighborhood $\cU$ with
$\delta_{x'} f = n$ for all $x' \in \cU$. The universal openness and
flatness assertions then follow from the corresponding statements of
Proposition \ref{prop:rel-dim-smoothing}.
\end{proof}

\begin{cor}\label{cor:rel-dim-smoothing-2-stacks} Suppose that $\cY$ is 
irreducible, and there exist $x \in \cX$
and $\cY' \subseteq \cY$ closed and irreducible containing $f(x)$ and
with support strictly smaller than $\cY$, such that 
\begin{Ilist}
\itm $f$ has universal relative dimension at least $n$ at $x$;
\itm every irreducible component $\cX'$ of $f^{-1}(\cY')$ containing $x$ has
$$\dim \cX'_{f(\eta)} - \codim_{\cY'} \overline{f(\cX')} = n,$$
where $\eta$ is the generic point of $\cX'$. 
\end{Ilist}
Then for every irreducible component $\cX''$ of $\cX$ containing 
$x$, we have
$$f(\cX'') \not\subseteq \cY'.$$
If further we have
\begin{Ilist}[2]
\itm $\cY'$ has codimension $c$ in $\cY$, and the inclusion 
$\cY' \hookrightarrow \cY$ has universal relative dimension at least $-c$,
\end{Ilist}
then
$$\dim \cX''_{f(\eta')} - \codim_{\cY} \overline{f(\cX'')} = n,$$
where $\eta'$ is the generic point of $\cX''$.
\end{cor}

\begin{proof} In light of Proposition \ref{prop:rel-dim-stacks-agree},
the proof of Proposition \ref{prop:rel-dim-smoothing-2} goes through 
verbatim in the stack setting, using Proposition \ref{prop:stack-codim-comp}
for the first part, and Propositions \ref{prop:stack-codim-add} and
\ref{prop:dim-formula-stacks} for the second.
\end{proof}

We next observe that composition of morphisms still behaves well in the
stack setting.

\begin{cor}\label{cor:rel-dim-composition-stacks} 
Suppose that $f:\cX \to \cY$ has
universal relative dimension at least $m$, and $g:\cY \to \cZ$ has 
universal relative dimension at least $n$. Then $g \circ f$ has universal
relative dimension at least $m+n$.
\end{cor}

\begin{proof} Let $P:Z \to \cZ$, $P':Y \to \cY \times_{\cZ} Z$ and
$P'':X \to \cX \times_{\cY} Y$ give smooth presentations of $g$ and $f$.
Then using that 
$$\cX \times_{\cY} Y = (\cX \times_{\cZ} Z) \times_{\cY \times_{\cZ} Z} Y,$$
we have a morphism $\cX \times_{\cY} Y \to \cX \times_{\cZ} Z$ which is
likewise a smooth cover, and thus yields a smooth presentation of 
$g \circ f$ as well.
$$\begin{tikzcd}
X \arrow{r}{P''} \arrow{dr} & \cX \times_{\cY} Y \arrow{d}\arrow{r} &
\cX \times_{\cZ} Z \arrow{d} \arrow{r} & \cX \arrow{d}{f} \\
{} & Y \arrow{r}{P'}\arrow{dr} & \cY \times_{\cZ} Z \arrow{r}\arrow{d} &
\cY \arrow{d}{g} \\
{} & {} & Z \arrow{r}{P} & \cZ
\end{tikzcd}$$
The desired
statement is then immediate from Lemma \ref{lem:rel-dim-composition}.
\end{proof}

We conclude by giving the obvious generalization of Corollary
\ref{cor:rel-dim-examples} to the stack setting.

\begin{cor}\label{cor:rel-dim-examples-stacks} Suppose that $f:\cX \to \cY$ 
is a closed immersion, and $g:\cY \to \cZ$ is smooth of relative dimension 
$n$, with $\cZ$ locally of finite type over a universally catenary scheme. 
If either $\cZ$ is regular and every component
of $\cX$ has codimension at most $c$ in $\cY$, or $\cX$ may be expressed 
locally as an intersection of determinantal conditions with expected 
codimensions adding up to $c$, then $g\circ f$ has universal relative 
dimension at least $n-c$.

Alternatively, suppose that $f:\cX \to \cY$ is a morphism of smooth 
algebraic stacks locally of finite type over a scheme $S$, with $S$ 
universally catenary, and let $m$ and $n$ be the relative dimensions of
$\cX$ and $\cY$ over $S$, respectively. Then $f$ has universal relative 
dimension at least $m-n$.
\end{cor}

Here, when we say that $\cX$ may be expressed locally as an intersection
of determinantal conditions with expected codimension adding up to $c$,
we mean simply that there exists some smooth presentation of $\cY$ 
for which this is the case.

\begin{proof} For the first statement, it follows immediately from the
scheme case that $f$ has universal relative dimension at least $-c$,
and $g$ has universal relative dimension at least $n$, so the desired
assertion follows from Corollary \ref{cor:rel-dim-composition-stacks}.

The second statement can be reduced immediately to the scheme statement
given in Corollary \ref{cor:rel-dim-examples}.
\end{proof}

\bibliographystyle{amsalpha}
\bibliography{gen}

\newcommand{\noopsort}[1]{} \newcommand{\printfirst}[2]{#1}
  \newcommand{\singleletter}[1]{#1} \newcommand{\switchargs}[2]{#2#1}
\providecommand{\bysame}{\leavevmode\hbox to3em{\hrulefill}\thinspace}
\providecommand{\MR}{\relax\ifhmode\unskip\space\fi MR }
\providecommand{\MRhref}[2]{%
  \href{http://www.ams.org/mathscinet-getitem?mr=#1}{#2}
}
\providecommand{\href}[2]{#2}
\begin{thebibliography}{{Sta}13}

\bibitem[BLR91]{b-l-r}
Siegfried Bosch, Werner Lutkebohmert, and Michel Raynaud, \emph{Neron models},
  Springer-Verlag, 1991.

\bibitem[Eis95]{ei1}
David Eisenbud, \emph{Commutative algebra with a view toward algebraic
  geometry}, Graduate Texts in Mathematics, vol. 150, Springer-Verlag, 1995.

\bibitem[GD60]{ega1}
Alexander Grothendieck and Jean Dieudonn\'e, \emph{{\'E}l\'ements de
  g\'eom\'etrie alg\'ebrique: {I.} {L}e langage des sch\'emas}, Publications
  math\'ematiques de l'I.H.\'E.S., vol.~4, Institut des Hautes \'Etudes
  Scientifiques, 1960.

\bibitem[GD64]{ega41}
\bysame, \emph{{\'E}l\'ements de g\'eom\'etrie alg\'ebrique: {IV.} \'{E}tude
  locale des sch\'emas et des morphismes de sch\'emas, premi\'ere partie},
  Publications math\'ematiques de l'I.H.\'E.S., vol.~20, Institut des Hautes
  \'Etudes Scientifiques, 1964.

\bibitem[GD65]{ega42}
\bysame, \emph{{\'E}l\'ements de g\'eom\'etrie alg\'ebrique: {IV.} \'{E}tude
  locale des sch\'emas et des morphismes de sch\'emas, seconde partie},
  Publications math\'ematiques de l'I.H.\'E.S., vol.~24, Institut des Hautes
  \'Etudes Scientifiques, 1965.

\bibitem[GD66]{ega43}
\bysame, \emph{{\'E}l\'ements de g\'eom\'etrie alg\'ebrique: {IV.} \'{E}tude
  locale des sch\'emas et des morphismes de sch\'emas, troisi\`eme partie},
  Publications math\'ematiques de l'I.H.\'E.S., vol.~28, Institut des Hautes
  \'Etudes Scientifiques, 1966.

\bibitem[GD67]{ega44}
\bysame, \emph{{\'E}l\'ements de g\'eom\'etrie alg\'ebrique: {IV.} \'{E}tude
  locale des sch\'emas et des morphismes de sch\'emas, quatri\'eme partie},
  Publications math\'ematiques de l'I.H.\'E.S., vol.~32, Institut des Hautes
  \'Etudes Scientifiques, 1967.

\bibitem[Hoc75]{ho1}
Melvin Hochster, \emph{Big {C}ohen-{M}acaulay modules and algebras and
  embeddability in rings of {W}itt vectors}, Queen's Papers on Pure and Applied
  Math \textbf{42} (1975), 106--195.

\bibitem[LMB00]{l-m-b}
G\'erard Laumon and Laurent Moret-Bailly, \emph{Champs algebriques},
  Springer-Verlag, 2000.

\bibitem[Oss]{os20}
Brian Osserman, \emph{Limit linear series moduli stacks in higher rank}, in
  preparation.

\bibitem[OT]{o-t2}
Brian Osserman and Montserrat {Teixidor i Bigas}, \emph{Linked symplectic forms
  and limit linear series in rank $2$ with special determinant}, in
  preparation.

\bibitem[Ser65]{se4}
Jean-Pierre Serre, \emph{{Alg\`ebre locale. Multiplicit\'es}}, Lecture Notes in
  Mathematics, no.~11, Springer-Verlag, 1965.

\bibitem[{Sta}13]{stacks-proj}
The {Stacks Project Authors}, \emph{Stacks project},
  http://stacks.math.columbia.edu, 2013.

\end{thebibliography}

\end{document}